\documentclass[conference]{IEEEtran}
\usepackage{amssymb}
\usepackage{graphicx}
\usepackage[tbtags]{amsmath}
\usepackage{amsfonts}
\usepackage{mathrsfs}

\newtheorem{thm}{Theorem}
\newtheorem{theorem}[thm]{Theorem}

\newtheorem{definition}[thm]{Definition}

\newtheorem{remark}{Remark}
\newtheorem{proof}{Proof:}

\begin{document}
\title {Generalized Estimating Equation for the\\ Student-t Distributions}
\author{
\IEEEauthorblockN{Atin Gayen}
\IEEEauthorblockA{Discipline of Mathematics\\
Indian Institute of Technology Indore \\
Indore, Madhya Pradesh 453552, India\\
Email: atinfordst@gmail.com}
\and
\IEEEauthorblockN{M. Ashok Kumar}
\IEEEauthorblockA{Discipline of Mathematics\\
Indian Institute of Technology Indore \\
Indore, Madhya Pradesh 453552, India\\
Email: ashokm@iiti.ac.in}
}

\maketitle

\begin{abstract}
In \cite{KumarS15J2}, it was shown that a generalized maximum likelihood estimation problem on a (canonical) $\alpha$-power-law model ($\mathbb{M}^{(\alpha)}$-family) can be solved by solving a system of linear equations. This was due to an orthogonality relationship between the $\mathbb{M}^{(\alpha)}$-family and a linear family with respect to the relative $\alpha$-entropy (or the $\mathscr{I}_\alpha$-divergence). Relative $\alpha$-entropy is a generalization of the usual relative entropy (or the Kullback-Leibler divergence). $\mathbb{M}^{(\alpha)}$-family is a generalization of the usual exponential family. In this paper, we first generalize the $\mathbb{M}^{(\alpha)}$-family including the multivariate, continuous case and show that the Student-t distributions fall in this family. We then extend the above stated result of \cite{KumarS15J2} to the general $\mathbb{M}^{(\alpha)}$-family. Finally we apply this result to the Student-t distribution and find generalized estimators for its parameters.
\end{abstract}

\section{Introduction and preliminaries}
\label{sec:introduction}
 The {\em exponential families} of probability distributions are important in statistics as many important probability distributions like Binomial, Poission, Gaussian and so on fall
in this class. Let $\bf{X}$ $=(X_1,\dots, X_d)$ be a $d$-dimensional random vector that follows a probability distribution
$p_\theta,\theta\in \Theta$, where $\Theta$ is an open subset of $\mathbb{R}^k$. Suppose also that $X_1,\dots,X_d$ are jointly continuous (or jointly discrete).
The family of probability distributions $\mathcal{E} = \lbrace p_\theta : \theta\in\Theta \rbrace$
is said to belong to a {\em $k$-parameter exponential family} if it can be represented in the following form \cite[Eq. (7.7.5)]{HoggCM13B}.
\begin{eqnarray}
\label{eqn:expoential_family}
	\hspace*{-0.03cm}{p_\theta(\textbf{\textit{x}})} = \left\{
	\begin{array}{ll}
	
	 {\exp [q(\textbf{\textit{x}}) +Z(\theta)+ w(\theta)^T f(\textbf{\textit{x}})]} &\hbox{if~} 
	\textbf{\textit{x}}\in \mathbb{S}\\
	  {0} &\hbox{otherwise,}
	\end{array}
	\right.
	\end{eqnarray}
where $\mathbb{S}$
denotes the support of $p_\theta$ (that is, $p_\theta(\textbf{\textit{x}}) > 0$ for $\textbf{\textit{x}}\in \mathbb{S}$). Here $w := (w_1,\dots,w_s)^T$ and 
$f := (f_1,\dots,f_s)^T$ such that for all $i=1,\dots,s$,
$w_i: \Theta\to \mathbb{R}$ and
$f_i: \mathbb{R}^d\to \mathbb{R}$  are some functions, and $q:\mathbb{R}^d\to \mathbb{R}$ is a 
non-negative function. Further, $Z:\Theta\to\mathbb{R}$ is a
function which makes $p_\theta$ a probability distribution and all $w_i(\theta)$ and $Z(\theta)$ are assumed to be differentiable on $\Theta$.

The exponential family can be thought of as {\em projections} of the
 Kullback-Leibler (KL) divergence on a set of probability distributions determined by some linear constraints, called {\em linear family} \cite{Csiszar75J}, \cite{CsiszarS04B}. $\mathscr{I}_\alpha$-divergence is a generalization of the KL-divergence and is defined as follows. For probability distributions $p$ and $q$ on $\mathbb{R}^d$,
 	\begin{eqnarray*}
	\label{1defn:I_alpha_divergence}
	\mathscr{I}_\alpha(p, q)
	 := \tfrac{\alpha}{1 - \alpha}\log\int 
	p (\textbf{\textit{x}})q(\textbf{\textit{x}})^{\alpha -1}d\textbf{\textit{x}}
	 - \tfrac{1}{1 - \alpha}\log\int p(\textbf{\textit{x}})^{\alpha}d\textbf{\textit{x}}\\ 
	 + \log\int  q(\textbf{\textit{x}})^{\alpha}d\textbf{\textit{x}},
	\end{eqnarray*}
 \cite{Sundaresan02ISIT}, \cite{Sundaresan07J}, \cite{LutwakYZ05J}, \cite{FujisawaE08J} (also known as {\em relative $\alpha$-entropy} \cite{KumarS15J1}, \cite{KumarS15J2}, {\em logarithmic density power divergence} \cite{MajiGB16J},
	{\em projective power divergence \cite{EguchiK10J}}, {\em $\gamma$-divergence} \cite{FujisawaE08J},
	\cite{CichockiA10J}). Here $\alpha>0,\alpha\neq 1$, called the order of the divergence. Notice that, $\mathscr{I}_\alpha$ coincides with the KL-divergence
	as $\alpha\to 1$ \cite{KumarS15J2}, \cite{CichockiA10J}. In this sense $\mathscr{I}_\alpha$-divergence can be regarded as a generalization of KL-divergence.  $\mathscr{I}_\alpha$-divergence also arises in information theory as redundancy in the mismatched case of guessing (for $\alpha < 1$) \cite{Sundaresan02ISIT}, source coding \cite{KumarS15J2}, and in encoding of tasks \cite{BunteL14J}.
	
	Analogous to the fact that the projections of KL-divergence on linear families yield an exponential family, the projections of $\mathscr{I}_\alpha$-divergence on linear families yield an $\alpha$-power-law family, $\mathbb{M}^{(\alpha)}$. A general $\mathbb{M}^{(\alpha)}$-family can be defined, including the continuous and
	multivariate case, as follows.
	\begin{definition}
\label{1defn:general_M_alpha_family}
The family of probability distributions $\lbrace p_\theta : \theta\in\Theta \rbrace$ is said to belong to a {\em $k$-parameter $\mathbb{M}^{(\alpha)}$ family} if it can be written as
\begin{displaymath}
	{p_\theta(\textbf{\textit{x}})} = \left\{
	\begin{array}{ll}
	 {Z(\theta)^{-1}\big[ q(\textbf{\textit{x}})^{\alpha - 1} +
w(\theta)^T f(\textbf{\textit{x}}) \big]^{\frac{1}{\alpha - 1}}} {\text{~if}~} 
	\textbf{\textit{x}}\in \mathbb{S}\\
	  {0} \hspace*{5cm}\text{{otherwise}}
	,
	\end{array}
	\right.
	\end{displaymath}
that is,
\begin{equation}
\label{1eqn:form_of_general_M_alpha_family}
p_\theta(\textbf{\textit{x}}) = Z(\theta)^{-1}\big[ q(\textbf{\textit{x}})^{\alpha - 1} +
w(\theta)^T f(\textbf{\textit{x}}) \big]_+^{\frac{1}{\alpha - 1}},
\end{equation}
with $[r]_+:=\max\lbrace r,0\rbrace$, for any $r\in\mathbb{R}$ and the functions
$w,f,q$ and $Z$ are as defined in (\ref{eqn:expoential_family}). 
If the number of $w_i$'s is
equal to that of $\theta_i$'s and each $w_i(\theta) = \theta_i$, such a
family is called {\em canonical\footnote{Analogous to the canonical exponential
 family.} 
$\mathbb{M}^{(\alpha)}$-family} \cite[Def. 8]{KumarS15J2}. This 
canonical form of the family arises as a projection of the 
$\mathscr{I}_\alpha$-divergence on a linear family of
probability distributions \cite{KumarS15J1}.
 Many well-known distributions such as {\em Wigner semi-circle distribution},
{\em Wigner parabolic distribution} and more interestingly, the {\em Student-t
distributions} fall in the class $\mathbb{M}^{(\alpha)}$.
\end{definition}
As KL-divergence is closely related to the maximum likelihood estimation (MLE),
the $\mathscr{I}_\alpha$-divergence is closely related to a robustified version of
MLE. Indeed, if $\textbf{X}_1,\dots, \textbf{X}_n$ is an independent and identically distributed (i.i.d.) sample drawn according to some $p_\theta$ of a parametric model, $\Pi = \{p_\theta : \theta\in\Theta\}$
and $\mathbb{S}$ is the common support of $\Pi$ (that is, support of members of $\Pi$ does not depend on $\theta$),
to find the MLE of $\theta$, one needs to solve the so-called {\em score equation} or {\em estimating equation} for $\theta$, given by
\begin{eqnarray}
\label{1eqn:score_equation_mle_in_terms_of_sample}
\frac{1}{n}\sum\limits_{j=1}^n s(\textbf{X}_j; \theta) = 0.
\end{eqnarray}
Here $s(\cdot ;\theta) := \nabla\log p_\theta(\cdot)$, called the {\em score function} and $\nabla$ stands for gradient with respect to $\theta$. 
If there is contamination in the observed sample, one modifies the score equation
by replacing the usual average of the score functions $s(\textbf{X}_j; \theta)$
in (\ref{1eqn:score_equation_mle_in_terms_of_sample}) by some weighted average
that down-weights the effect of the outliers. Motivated by the works of Field and Smith \cite{FieldS94J} and Windham \cite{Windham95J}, the following
estimating equation was proposed by Jones {\em et. al.} \cite{JonesHHB01J}:
\begin{eqnarray}
\label{1eqn:score_equation_I_alpha}
\dfrac{\frac{1}{n}\sum\limits_{j=1}^n  p_\theta(\textbf{X}_j)^{\alpha-1}
s(\textbf{X}_j;\theta)}
{\frac{1}{n}\sum\limits_{j=1}^n  p_\theta(\textbf{X}_j)^{\alpha-1}}
= \dfrac{\int p_\theta(\textbf{\textit{x}})^\alpha s(\textbf{\textit{x}};\theta)d\textbf{\textit{x}}}
{\int p_\theta(\textbf{\textit{x}})^\alpha d\textbf{\textit{x}}},
\end{eqnarray} 
where $\alpha>1$. The above equation was proposed based on the following intuition. If a sample point 
$\textbf{\textit{x}}$ is not compatible to the true distribution $p_\theta$,
then $p_\theta(\textbf{\textit{x}})^{\alpha -1}$ would be smaller and thus
down-weights the effect of $\textbf{\textit{x}}$ in the average of
the score functions. The equation is obtained by equating the normalized empirical
weighted average to its hypothetical one (c.f. \cite{BasuSP11B}).
Observe that, (\ref{1eqn:score_equation_I_alpha}) does not make sense in terms
of robustness for $\alpha < 1$. However, it is a valid estimation problem
even for $\alpha <1$ as minimization of $\mathscr{I}_\alpha$-divergence for
$\alpha <1$ corresponds to certain estimation problems in information theory, such as guessing \cite{Sundaresan07J},
source coding (see \cite[Sec. II C]{KumarS15J2}) and encoding of tasks \cite{BunteL14J}.

Csisz\'ar and Shields showed that if $\Pi$ is a canonical exponential family with
finite support $\mathbb{S}$, the MLE (if exists and unique) is a solution to a system of linear equations
\cite[Th. 3.3]{CsiszarS04B}. This was due to an orthogonality relationship between the exponential family and a linear family with respect to the relative entropy. By exploiting the geometry between $\mathscr{I}_\alpha$-divergence, $\mathbb{M}^{(\alpha)}$-family and 
linear family, analogously, Kumar and Sundaresan showed that if $\Pi$ is a canonical $\mathbb{M}^{(\alpha)}$-family with finite support $\mathbb{S}$, then the solution of
(\ref{1eqn:score_equation_I_alpha}) (if exists and unique) is same as solution of
a system of linear equations \cite[Th. 18 and Th. 21]{KumarS15J2}. In this paper, we solve this problem for the general $\mathbb{M}^{(\alpha)}$-family by directly solving the estimating equation. We show that, under some regularity assumptions,
the result continues to hold for general $\mathbb{M}^{(\alpha)}$-family as well. We then apply this result to find estimators for the student-t distributions.
We assume the following regularity conditions unless stated otherwise.
\begin{itemize}
\item[(a)]All the integrals are well-defined over $\mathbb{S}$ with respect to the Lebesgue measure
on $\mathbb{R}^{d}$ in the continuous case and with respect to the counting measure in
discrete case.
\item[(b)]For probability distributions $p_\theta,p_\eta\in \Pi$, if $\theta\neq \eta$, then $p_\theta \neq p_\eta$ on a set of positive measure.
 \item[(c)] The support $\mathbb{S}$ does not depend on $\theta$. Integration with respect to $\textbf{\textit{x}}$ and
differentiation with respect to $\theta$ can be interchanged.
\end{itemize} 
 \section{Estimation on $\mathbb{M}^{(\alpha)}$ family}
 \label{sec:Estimation on M alpha family}
 In this section, we solve the estimation problems (MLE or the generalized MLE) associated with
 the KL-divergence and $\mathscr{I}_\alpha$-divergence by solving the respective estimating equations.
 \begin{theorem}
The following are true.
\vspace*{-0.2cm}
\label{1thm:estimating_equations_for_general_families}
\begin{itemize}
\item[(i)] The MLE on an exponential family as defined in (\ref{eqn:expoential_family}) must satisfy 
 \begin{equation}
\label{1eqn:estimating_equation_for_general_exponential_family}
\partial _r [w (\theta)]^T \mathbb{E}_\theta [f(\textbf{X})] = 
\partial _r [w (\theta)]^T \bar{f},~r=1,\dots,k.
\end{equation}
\item[(ii)] The Jones et. al. estimator on an $\mathbb{M}^{(\alpha)}$ family as in
Definition \ref{1defn:general_M_alpha_family} must satisfy
\begin{align}
\label{1eqn:estimating_equation_for_general_M_alpha_family}
\frac{ \partial_r [w(\theta)]^T \mathbb{E}_\theta [f(\textbf{X})]}
{\mathbb{E}_\theta [q(\textbf{X})^{\alpha - 1} + 
w(\theta)^T f(\textbf{X}) ]}=
\frac{ \partial_r [w(\theta)]^T \bar{f}}
{\overline{q^{\alpha -1}} +  w(\theta)^T \bar{f}},\\r=1,\dots,k\nonumber.
\end{align}
\end{itemize}
Here $\partial_r$ denotes the partial derivative with respect to $\theta_r$,
$\bar{f} := (\bar{f_1},\dots,\bar{f_s})^T$, $\bar{f_i}=\tfrac{1}{n}\sum_{j=1}^n
f_i(\textbf{X}_j)$ for $i=1,\dots,s$, and $\overline{q^{\alpha -1}}:= \tfrac{1}{n}\sum_{j=1}^n q(\textbf{X}_j)^{\alpha -1}$.
\end{theorem}
\begin{proof} 
Consider an i.i.d. sample $\textbf{X}_1,\dots,
\textbf{X}_n$ drawn according to some $p_\theta$, where $\textbf{X}_i = (X_{1i},\dots, X_{di})^T, ~\forall i= 1,\dots, n$. 
\begin{itemize}
	\item [(i)] If $p_\theta$ is in an exponential family as in
	(\ref{eqn:expoential_family}), we have for $r=1,\dots,k$,
	\begin{align*}
	\partial_r[\log p_\theta (\textbf{\textit{x}})] = \partial_r [Z(\theta)] + \partial_r [w(\theta)]^T
	f(\textbf{\textit{x}}). 
	\end{align*}
	$\mathbb{E}_\theta [\partial_r \log p_\theta (\textbf{X})]=0$, by the
	regularity condition (c), thus
	we have
	\begin{equation}
	\label{1eqn:expectation_of_exponential_family}
	\partial_r [Z(\theta)] + \partial_r [w(\theta)]^T
	\mathbb{E}_\theta [f(\textbf{X})]=0.
	\end{equation}
	Using (\ref{1eqn:expectation_of_exponential_family})
	in the estimating equation (\ref{1eqn:score_equation_mle_in_terms_of_sample}), we have (\ref{1eqn:estimating_equation_for_general_exponential_family}).
	
\item[(ii)]The estimating equation (\ref{1eqn:score_equation_I_alpha}) can be re-written as
		\begin{eqnarray}
		\label{1eqn:score_equation_I_alpha_calculated_2}
		\tfrac{\frac{1}{n}\sum\limits_{i=1}^n p_\theta(X_i)^{\alpha-2}\nabla p_\theta(X_i)}{\frac{1}{n}\sum\limits_{i=1}^n p_\theta(X_i)^{\alpha-1}} =
		\tfrac{\int p_\theta(x)^{\alpha-1}\nabla p_\theta(x)d\textbf{\textit{x}}}{\int p_\theta(x)^{\alpha}d\textbf{\textit{x}}}.
		\end{eqnarray}

 Since $p_\theta\in \mathbb{M}^{(\alpha)}$, from
(\ref{1eqn:form_of_general_M_alpha_family}) we have
\begin{eqnarray}
\lefteqn{p_\theta(\textbf{\textit{x}})^{\alpha -2} 	
	\partial_r [p_\theta (\textbf{\textit{x}})]}\nonumber\\
&& = Z(\theta)^{1-\alpha} \big\lbrace Z(\theta)^{-1} \partial _r[Z(\theta)]
[q(\textbf{\textit{x}})^{\alpha -1} +\nonumber\\
&&\hspace*{1.2cm} w(\theta)^T f(\textbf{\textit{x}})] -\tfrac{1}{\alpha -1} \partial_r[w(\theta)]^T f(\textbf{\textit{x}})
\big\rbrace,\nonumber
\end{eqnarray}
for $r=1,\dots,k$.
Using this in (\ref{1eqn:score_equation_I_alpha_calculated_2}) we get,
\begin{eqnarray*}
Z(\theta)^{-1}\partial_r[Z(\theta)] -  \tfrac{\tfrac{1}{\alpha -1} \partial_r[w(\theta)]^T
\big[\tfrac{1}{n}\sum\limits_{j=1}^n 
f(\textbf{X}_j)\big]}{\tfrac{1}{n}\sum\limits_{j=1}^n [q(\textbf{X}_j)^{\alpha -1} +
w(\theta)^T f(\textbf{X}_j)]}\hspace*{1cm}\\
=
Z(\theta)^{-1}\partial_r[Z(\theta)] - \tfrac{\tfrac{1}{\alpha -1} \partial_r[w(\theta)]^T
\int p_\theta (\textbf{\textit{x}})
f(\textbf{\textit{x}})d\textbf{\textit{x}}}
{\int p_\theta (\textbf{\textit{x}})^\alpha d\textbf{\textit{x}}},
\end{eqnarray*}
which implies
(\ref{1eqn:estimating_equation_for_general_M_alpha_family}). 
\end{itemize}
\end{proof}

In the following theorem we show that for a {\em regular} $\mathcal{E}$ or $\mathbb{M}^{(\alpha)}$-
family, Theorem \ref{1thm:estimating_equations_for_general_families} can be improved further.
\begin{definition}
\label{1defn:regular_family}
If, for an exponential family $\mathcal{E}$ (respectively, $\mathbb{M}^{(\alpha)}$-family), the following conditions are true, then such a family is said to be a 
{\em $k$-parameter regular exponential family}
(respectively, {\em $k$-parameter regular $\mathbb{M}^{(\alpha)}$ family}).
\begin{itemize}
\item[(i)] Number of $w_i$'s and number of $\theta_i$'s are equal ($s=k$),
\item[(ii)] 1, $w_1(\cdot), \dots, w_k(\cdot)$ are functionally independent
on $\Theta$,
\item[(iii)] $f_1(\cdot),\dots,f_k(\cdot)$ are functionally independent on $\mathbb{S}$.
\end{itemize}
\end{definition}
\begin{theorem} 
 \label{1thm:estimating_equation_on_regular_family}
If the families in Theorem \ref{1thm:estimating_equations_for_general_families} are further regular,
then
\begin{itemize}
\item[(i)] the MLE on exponential family must satisfy
\begin{equation}
\label{1eqn:estimating_equation_for_regular_exponential}
\mathbb{E}_\theta [f(\textbf{X})] = \bar{f},
\end{equation}
\item[(ii)] the Jones {\em et. al.} estimator on an $\mathbb{M}^{(\alpha)}$-family must satisfy
\begin{equation}
\label{1eqn:estimating_equation_for_regular_M_alpha}
\frac{\mathbb{E}_\theta [f(\textbf{X})]}
{\mathbb{E}_\theta [q(\textbf{X})^{\alpha - 1} ]} = 
\frac{ \bar{f}}
{\overline{q^{\alpha -1}}}.
\end{equation}
\end{itemize}
\end{theorem}
\begin{proof}
Since for a $k$-parameter regular family, $1$, $w_1(\cdot),\dots,w_k(\cdot)$ are functionally independent
on $\Theta$. This implies the vectors $\big(\partial_1[w_i(\theta)],\dots,\partial_k[w_i(\theta)]\big)$,
 $i=1,\dots,k$, are linearly independent for every $\theta\in\Theta$, because if
\begin{equation*}
\sum\limits_{i=1}^k c_i\big(\partial_1[w_i(\theta)],\dots,\partial_k[w_i(\theta)]\big)  =\bf{0},
\end{equation*}
for some scalars $c_1,\dots,c_k$, then
\begin{equation*}
c_1w_1(\theta) +\cdots +c_k w_k(\theta) = m,
\end{equation*}
where $m$ is a constant. Then functional independence of
$1$, $w_1(\cdot),\dots,w_k(\cdot)$ implies that $m=c_1=\cdots =c_k
=0$. 
\begin{itemize}
\item[(i)] (\ref{1eqn:estimating_equation_for_general_exponential_family}) can be re-written as
\begin{equation}
\label{1eqn:estimatinag_equation_on_E_and_B_alpha_as_a_system}
\partial _r [w (\theta)]^T \big(\mathbb{E}_\theta [f(\textbf{X})] - \bar{f} \big) = 0.
\end{equation}
For any $\theta\in \Theta$, this is a system of $k$-homogeneous equations in $k$ unknowns
where the coefficient matrix is given by
\begin{equation*}
D:= \big(\partial_i[w_j(\theta)]\big)_{k\times k}.
\end{equation*}
As all the columns of $D$ are linearly independent, the determinant of $D$ is non-zero. Hence
(\ref{1eqn:estimatinag_equation_on_E_and_B_alpha_as_a_system}) implies
$\mathbb{E}_\theta [f_i(\textbf{X})] - \bar{f_i} = 0$, for $i=1,\dots,k$. By the regularity assumption
(b), the estimator must satisfy (\ref{1eqn:estimating_equation_for_regular_M_alpha}).

\item [(ii)] (\ref{1eqn:estimating_equation_for_general_M_alpha_family})
can be re-written as
\begin{equation*}
 \partial_r [w(\theta)]^T \Big[\tfrac{ \mathbb{E}_\theta [f(\textbf{X})]}
{\mathbb{E}_\theta [q(\textbf{X})^{\alpha - 1} + 
	w(\theta)^T f(\textbf{X}) ]} -
\tfrac{\bar{f}}
{\overline{q^{\alpha -1}} +   w(\theta)^T \bar{f}}\Big]=0.
\end{equation*}
As the family is regular, the above equation becomes
\begin{equation*}
\frac{ \mathbb{E}_\theta [f(\textbf{X})]}
{\mathbb{E}_\theta [q(\textbf{X})^{\alpha - 1} + 
	w(\theta)^T f(\textbf{X}) ]} -
\frac{\bar{f}}
{\overline{q^{\alpha -1}} + w(\theta)^T \bar{f}}=0,
\end{equation*}
by a similar argument as in (i). 
Now,
\begin{eqnarray}
\lefteqn{\overline{q^{\alpha -1}} + w(\theta)^T \bar{f}}\nonumber\\
&& =
\overline{q^{\alpha -1}} + 
\tfrac{\overline{q^{\alpha -1}} +  w(\theta)^T \bar{f}}
{\mathbb{E}_\theta [q(\textbf{X})^{\alpha - 1} + 
	w(\theta)^T f(\textbf{X}) ]} 
w(\theta)^T \mathbb{E}_\theta [f(\textbf{X})].\nonumber
\end{eqnarray}
Thus,
\begin{eqnarray}
\lefteqn{\big\lbrace\overline{q^{\alpha -1}} +  w(\theta)^T \bar{f}\big\rbrace\big\lbrace\mathbb{E}_\theta [q(\textbf{X})^{\alpha - 1} + 
	w(\theta)^T f(\textbf{X}) ]\big\rbrace}\nonumber\\
&& =
\overline{q^{\alpha -1}}~\mathbb{E}_\theta [q(\textbf{X})^{\alpha - 1} + 
w(\theta)^T f(\textbf{X}) ]\nonumber\\
&& \hspace*{1cm}+
[\overline{q^{\alpha -1}} + w(\theta)^T \bar{f}]
w(\theta)^T  \mathbb{E}_\theta [f(\textbf{X})],\nonumber
\end{eqnarray}
that is,
\begin{eqnarray}
\lefteqn{[\overline{q^{\alpha -1}} +  w(\theta)^T \bar{f}]~\mathbb{E}_\theta [q(\textbf{X})^{\alpha - 1}]}\nonumber\\
&&\hspace{1cm} =
\overline{q^{\alpha -1}}~\mathbb{E}_\theta [q(\textbf{X})^{\alpha - 1} + 
w(\theta)^T f(\textbf{X}) ].\nonumber
\end{eqnarray}
This implies 
\begin{equation*}
\frac{\overline{q^{\alpha -1}}}{\mathbb{E}_\theta [q(\textbf{X})^{\alpha - 1}]}
=
\frac{\overline{q^{\alpha -1}} + w(\theta)^T \bar{f}}{\mathbb{E}_\theta [q(\textbf{X})^{\alpha - 1} + 
	w(\theta)^T f(\textbf{X}) ]}.
\end{equation*}
Thus the estimator must satisfy (\ref{1eqn:estimating_equation_for_regular_M_alpha}).
\end{itemize}
\end{proof}
\begin{remark}
Observe that, Theorem \ref{1thm:estimating_equation_on_regular_family} (ii) essentially extends the
result known for a canonical $\mathbb{M}^{(\alpha)}$-family with finite support 
\cite[Th. 18 and Th. 21]{KumarS15J2}.
\end{remark}
 
 \section{Generalized Maximum Likelihood Estimation on Student-t distributions}
In this section, we first show that the Student-t distributions form an $\mathbb{M}^{(\alpha)}$-family
and then we apply Theorem \ref{1thm:estimating_equations_for_general_families}
to find the estimator for their parameters.

The $d$-dimensional Student-t distribution with mean $\mu := (\mu_1,\dots,\mu_d)^T$ and 
a positive-definite covariance matrix $\Sigma:=(\sigma _{ij})_{d\times d}$ is given by 
		\begin{equation}
		\label{1eqn:student_distribution}
		p_\theta (\textbf{\textit{x}}) = N_{\theta,\alpha} \big[ 1 +  b_\alpha (\textbf{\textit{x}}
		- {\mu})^T \Sigma ^{-1} (\textbf{\textit{x}} - \bf{\mu}) \big]_+ ^{\frac{1}{\alpha - 1}},
		\end{equation}
		where $d/(d+2) < \alpha, \alpha \neq 1$ and $b_\alpha = [1-\alpha]/[2\alpha - d (1 - \alpha)]$.
		Here $\nu:= [b_\alpha+2(1-\alpha)^2]/(1-\alpha)^2$ is the degrees of freedom. Let $\Sigma ^{-1}
		:= (\sigma^{ij})$, the inverse of $\Sigma$.
		The normalizing constant $N_{\theta,\alpha}$ is given by
		\begin{displaymath}
		N_{\theta, \alpha} := \left\{
		\begin{array}{ll}
		{\frac{b_\alpha^{d/2}\Gamma (1/1-\alpha)}{\Gamma ([1/1-\alpha] - [d/2]) \pi ^{d/2}
				|\Sigma|^{1/2}} } &\hbox{~if~} 
		\alpha < 1\\\\
		{\frac{[-b_\alpha]^{d/2}\Gamma ([\alpha/\alpha - 1] + [d/2])}{\Gamma (\alpha/\alpha - 1) \pi ^{d/
					2} |\Sigma|^{1/2}} } &\hbox{~if~}
		\alpha > 1.
		\end{array}
		\right.
		\end{displaymath}
		The support of this distribution is given by
		\begin{displaymath}
		\mathbb{S} = \left\{
		\begin{array}{ll}
		{\mathbb{R}^d} &\hbox{~if~} 
		\alpha < 1\\
		{\lbrace \textbf{\textit{x}} : (\textbf{\textit{x}}
			- {\mu})^T \Sigma ^{-1} (\textbf{\textit{x}} - {\mu}) \geq -1/b_\alpha\rbrace  } &\hbox{~if~}
		\alpha > 1.
		\end{array}
		\right.
		\end{displaymath}
		We use the following notations. For a matrix $A=(a_{ij})_{d\times d}$,
		\begin{eqnarray*}
			\text{Tr}(A) := \text{Trace of } A = \sum_i a_{ii},~~ A^{-T}:= (A^{-1})^T,\\
			\text{Vec} (A) := (a_{11},\dots,a_{1d},a_{21},\dots,a_{2d},\dots,a_{d1},\dots,a_{dd})^T,
		\end{eqnarray*}
		that is, if $A$ is a matrix of order $d$, then $\text{Vec}(A)$ is a $d^2$-dimensional
		column vector whose $[(i-1)d + j]$-th entry is $a_{ij}$, for all $i,j=1,\dots,d$.
		For $\textbf{\textit{x}}\in \mathbb{S}$, we can re-write
		(\ref{1eqn:student_distribution}) as
		\begin{align}
		\label{1eqn:student_distribution_as_M_alpha}
		p_\theta (\textbf{\textit{x}})
		 &= N_{\theta,\alpha} \big[ 1 + b_\alpha \lbrace
		\textbf{\textit{x}}^T \Sigma ^{-1} \textbf{\textit{x}} - 2 {\mu}^T\Sigma ^{-1} \textbf{\textit{x}}+\mu^T \Sigma ^{-1}\mu\rbrace\big]^{\frac{1}{\alpha -1}}\nonumber\\
		&= N_{\theta,\alpha} \big[ 1 + b_\alpha \lbrace \text{Tr}(
		\Sigma ^{-T} \textbf{\textit{x}}\textbf{\textit{x}}^T) - 2 (\Sigma ^{-1}\mu)^T \textbf{\textit{x}}\nonumber\\
		&\hspace*{4.5cm}+\mu^T \Sigma ^{-1}\mu\rbrace\big]^{\frac{1}{\alpha -1}}\nonumber\\
		&= N_{\theta,\alpha} \big[ 1 + b_\alpha \lbrace 
		\text{Vec} (\Sigma ^{-1})^T \text{Vec}(\textbf{\textit{x}}\textbf{\textit{x}}^T)\nonumber\\
		&\hspace*{2.5cm}
		 - 2 (\Sigma ^{-1}\mu)^T \textbf{\textit{x}}
		+\mu^T \Sigma ^{-1}\mu\rbrace\big]^{\frac{1}{\alpha -1}}.
		\end{align}
		Comparing (\ref{1eqn:student_distribution_as_M_alpha})
		with (\ref{1eqn:form_of_general_M_alpha_family}), we have
		Student-t distributions as a $(d^2+d)$-parameter 
		$\mathbb{M}^{(\alpha)}$-family, with
		\begin{eqnarray*}
		\theta = (\mu,\Sigma^{-1}),~ Z(\theta)^{-1} = N_{\theta,\alpha},~q(\textbf{\textit{x}})\equiv 1,\\
			w_1(\theta) = b_\alpha(\mu^T \Sigma ^{-1}\mu), ~~ f_1(\textbf{\textit{x}})= 1,\\
			w_2(\theta) = - 2b_\alpha (\Sigma ^{-1}\mu)^T, ~~ f_2(\textbf{\textit{x}}) = \textbf{\textit{x}},\\
			w_3(\theta) = b_\alpha \text{Vec} (\Sigma ^{-1})^T, ~~ f_3(\textbf{\textit{x}}) =
			\text{Vec}(\textbf{\textit{x}}\textbf{\textit{x}}^T).
			\end{eqnarray*}

Since for $\alpha < 1$, the support $\mathbb{S}$
does not depend upon the parameters, we can apply Theorem \ref{1thm:estimating_equations_for_general_families} (ii) to estimate the parameters $\mu$ and $\Sigma$.
Let us first calculate the derivative of each $w_i(\theta)$ with respect to each parameter.
\begin{eqnarray*}
\partial_{\mu} [w_1(\theta)] = 2b_\alpha (\Sigma ^{-1}\mu),\\
\partial_{\mu} [w_2(\theta)] = -2b_\alpha (\Sigma^{-1}),\\
\partial_{\mu} [w_3(\theta)] = O_{d\times d^2},
\end{eqnarray*}
where $O_{d\times d^2}$ is the zero matrix of order $d\times d^2$. For $i,j=1,\dots,d$
\begin{eqnarray*}
\partial_{\sigma^{ij}} [w_1(\theta)] = 2b_\alpha \mu_i\mu_j ,\\
\partial_{\sigma^{ij}} [w_2(\theta)] = u_{ij}, \\
\partial_{\sigma^{ij}} [w_3(\theta)] = v_{ij},
\end{eqnarray*}
where $u_{ij}$ is a $d$-dimensional row vector whose entries
are zero except the
$i$-th and $j$-th entries which are $(-2b_\alpha\mu_j)$ and $(-2b_\alpha\mu_i)$, respectively. Similarly, $v_{ij}$ is a 
$d^2$-dimensional row vector whose entries are zero except
the $[(i-1)d+j]$-th and $[(j-1)d+i]$-th which are
equal to $b_\alpha$.

Consider now an i.i.d. sample $\textbf{X}_1,\dots, \textbf{X}_n$ according to a $p_\theta$ of the form
(\ref{1eqn:student_distribution}), where $\textbf{X}_i = (X_{1i},\dots,X_{di})^T$, $i=1,\dots,n$.
Define $\overline{\textbf{X}} = (\overline{X}_1,\dots,\overline{X}_d)^T$, where
$\overline{X}_i = \frac{1}{n}\sum_{l=1}^n X_{il}$, for $i=1,\dots,d$, and $\overline{\textbf{X}
\textbf{X}^T}$ is the matrix of order $d\times d$ whose $(i,j)$-th entry is $\frac{1}{n}\sum_{l=1}^n X_{il}X_{jl}$. Let us denote
\begin{eqnarray*}
\textbf{Y}:= 1 + b_\alpha [ 
  \text{Vec} (\Sigma ^{-1})^T \text{Vec}(\overline{\textbf{X}\textbf{X}^T})\\
&&\hspace*{-1.5cm} 
  - 2 (\Sigma ^{-1}\mu)^T 
 \overline{\textbf{X}} 
 + \mu^T \Sigma ^{-1}\mu].
\end{eqnarray*} 
Using (\ref{1eqn:estimating_equation_for_general_M_alpha_family}), we have the following estimating equations for the Student-t distributions. 
\begin{eqnarray*}
\hspace*{1cm}\tfrac{-2b_\alpha \Sigma^{-1} \mathbb{E}_\theta [\textbf{X}] + 2b_\alpha (\Sigma^{-1}\mu)}
{\mathbb{E}_\theta[\textbf{Y}]}
 = \tfrac{-2b_\alpha \Sigma^{-1} \overline{\textbf{X}} + 2b_\alpha (\Sigma^{-1}\mu)}
 {\textbf{Y}},\\
&&\hspace*{-8.8cm}\tfrac{2 b_\alpha \mu_i\mu_j + u_{ij} \mathbb{E}_\theta [\textbf{X}] + v_{ij} \mathbb{E}_\theta
 [\text{Vec}(\textbf{X}\textbf{X}^T)]}{\mathbb{E}_\theta[\textbf{Y}]}
 = \tfrac{2 b_\alpha \mu_i\mu_j + u_{ij} \overline{\textbf{X}} + v_{ij} \text{Vec}(\overline{\textbf{X}\textbf{X}^T})}{\textbf{Y}},
\end{eqnarray*}
for $i,j = 1,\dots,d$.
Since $\mathbb{E}_\theta [\textbf{X}] = \mu$ and $\mathbb{E}_\theta [X_iX_j] = \sigma_{ij} +\mu_i\mu_j$,
the above system reduces to
\begin{eqnarray*}
\mu  = \overline{\textbf{X}},~\text{and}~
\frac{1}{n}\sum\limits_{l=1}^n X_{il}X_{jl} = \mu_i\mu_j + \tfrac{\textbf{Y}}{\mathbb{E}_\theta[\textbf{Y}]} \sigma_{ij}.
\end{eqnarray*}
Using these we have
\begin{align}
\mathbb{E}_\theta[\textbf{Y}] 
 &=\mathbb{E}_\theta [ 1 + b_\alpha \lbrace 
  \text{Vec} (\Sigma ^{-1})^T \text{Vec}(\textbf{X}\textbf{X}^T)\nonumber\\
  &\hspace*{3.3cm} - 2 (\Sigma ^{-1}\mu)^T
 \textbf{X} + \mu^T \Sigma ^{-1}\mu\rbrace] \nonumber\\
 &= 1 + b_\alpha \lbrace 
  \text{Vec} (\Sigma ^{-1})^T [\text{Vec}(\Sigma) + \text{Vec}(\mu\mu^T)] \nonumber\\
  &\hspace*{3.3cm}- 2 (\Sigma ^{-1}\mu)^T
  \mu + \mu^T \Sigma ^{-1}\mu\rbrace \nonumber\\
  &= 1 + b_\alpha \lbrace 
  \text{Tr}(\Sigma ^{-1} \Sigma) + \text{Tr}(\Sigma^{-1}\mu \mu^T) - \mu^T \Sigma ^{-1}\mu\rbrace \nonumber\\
  &=  1 + b_\alpha \lbrace 
  d + \text{Tr}(\mu^T\Sigma^{-1}\mu) - \text{Tr}(\mu^T \Sigma ^{-1}\mu)\rbrace \nonumber\\
  &= 1+d\cdot b_\alpha,\label{1eqn:value_of_E[K]}
\end{align}
and
\begin{align}
\textbf{Y} &= 1 + b_\alpha \lbrace 
  \text{Vec} (\Sigma ^{-1})^T \text{Vec}(\overline{\textbf{X}\textbf{X}^T}) - 2 (\Sigma ^{-1}\mu)^T
 \overline{\textbf{X}} + \mu^T \Sigma ^{-1}\mu\rbrace \nonumber\\
 &= 1 + b_\alpha \lbrace 
  \text{Vec} (\Sigma ^{-1})^T \text{Vec}\big(\tfrac{\textbf{Y}}{\mathbb{E}_\theta[\textbf{Y}]}\Sigma + \mu\mu^T\big)\nonumber\\
  &\hspace*{4cm} - 2 (\Sigma ^{-1}\mu)^T \mu + \mu^T \Sigma ^{-1}\mu\rbrace  \nonumber\\
  &= 1 + b_\alpha \lbrace 
  \text{Vec} (\Sigma ^{-1})^T \text{Vec}\big(\tfrac{\textbf{Y}}{\mathbb{E}_\theta[Y]}\Sigma \big)\nonumber\\
  &\hspace*{3cm}+ 
  \text{Vec} (\Sigma ^{-1})^T \text{Vec}(\mu\mu^T) - \mu^T \Sigma ^{-1}
  \mu\rbrace \nonumber\\
  &= 1 + b_\alpha \lbrace
  \tfrac{\textbf{Y}}{\mathbb{E}_\theta[\textbf{Y}]} d + \text{Tr} (\Sigma^{-1}\mu\mu^T)  - \mu^T\Sigma^{-1}\mu\rbrace  \nonumber\\
  &= 1 + d\cdot b_\alpha\tfrac{\textbf{Y}}{\mathbb{E}_\theta[\textbf{Y}]}.\label{1eqn:value_of_K}
\end{align}
Equations (\ref{1eqn:value_of_E[K]}) and (\ref{1eqn:value_of_K}) together imply 
$\textbf{Y} = \mathbb{E}_\theta[\textbf{Y}]$. Thus the estimating equations for Student-t distributions become
\begin{eqnarray*}
\hspace*{0.7cm}\mu = \overline{\textbf{X}},~\text{and}~
\frac{1}{n}\sum\limits_{l=1}^n X_{il}X_{jl} = \mu_i\mu_j + \sigma_{ij}.
\end{eqnarray*}
Hence the estimators for $\mu$ and $\Sigma$ are
\begin{eqnarray}
\label{1eqn:estimators_for_student_distribution}
\widehat{\mu} = \overline{\textbf{X}},~\text{and}~
\widehat{\sigma_{ij}} = \frac{1}{n}\sum\limits_{l=1}^n X_{il}X_{jl} - \widehat{\mu_i}\widehat{\mu_j}.
\end{eqnarray}
These are summarized in the following theorem.
\begin{theorem}
\label{1thm:estimators_for_student_t}
Let $\alpha < 1$. The estimator for the mean and covariance parameters of a Student-t distribution in
(\ref{1eqn:student_distribution}) by the estimating equation (\ref{1eqn:score_equation_I_alpha}) of Jones {\em et.~al.}, are
\begin{equation}
\widehat{\mu} = \overline{\textbf{X}},~~ \widehat{\sigma_{ij}} = \frac{1}{n} \sum\limits_{l=1}^n
(X_{il} - \widehat{\mu_i}) (X_{jl} - \widehat{\mu_j}).
\end{equation}
\end{theorem}
\begin{remark}
\label{1rem:equivalence_of_student_t_and_normal_distribution}
It can be shown that as $\alpha\to 1$, the Student-t distributions coincide with a normal distribution
with mean $\mu$ and covariance matrix $\Sigma$. Also, for $\alpha =1$, the estimating equation
(\ref{1eqn:score_equation_I_alpha}) is actually the usual score equation
(\ref{1eqn:score_equation_mle_in_terms_of_sample}) of MLE. Thus there is a continuity on the $\alpha$-estimation for $\alpha$ in $(0,1]$.
\end{remark}

The condition that the support $\mathbb{S}$ is independent of the parameters is necessary for Theorem \ref{1thm:estimating_equations_for_general_families}. The
{\em uniform distribution}
in $(0, \theta)$, where $\theta$ is the unknown parameter,
can be expressed as an exponential family. But the
support $\mathbb{S}$ of this family depends on the parameter $\theta$. Also the MLE for $\theta$
can not be obtained by simply solving the estimating
equation (\ref{1eqn:estimating_equation_for_general_exponential_family}). Here we present such an example for the Jones {\em et. al.} estimation (\ref{1eqn:score_equation_I_alpha}). 

Let us consider, for simplicity, the Student-t distributions with $\alpha = 2$ and variance $\sigma=1$.
Then the pdf is given by
\begin{equation}
\label{1eqn:student_t_for_alpha_equal_to_2}
p_\mu (x) = N_{2} \Big[1 -\frac{(x -\mu)^2}{5}\Big]_+,
\end{equation}
where $N_2 = \Gamma(5/2)/\sqrt{5\pi} \Gamma (2) = 3/4\sqrt{5}$ and the support is given by
\begin{equation*}
\mathbb{S} = \lbrace x : \mu - \sqrt{5} \leq x \leq \mu +\sqrt{5}\rbrace
\end{equation*}
which depends on the unknown parameter $\mu$. Thus we cannot use Theorem \ref{1thm:estimating_equations_for_general_families} (ii) directly
to estimate $\mu$. However, solving the estimating equation (\ref{1eqn:score_equation_I_alpha})
is same as maximizing the following generalized likelihood function\footnote{This
coincides with the usual log likelihood function for MLE as $\alpha\to 1$.} for $p_\theta$,
\begin{equation}
\label{1eqn:likelihood_function_for_I_alpha}
L^{(\alpha)}(\theta) := \tfrac{\alpha}{\alpha-1}\log\Big[\tfrac{1}{n}\sum\limits_{j=1}^n p_\theta(\textbf{X}_j)^{\alpha-1}\Big] - \log\Big[\int p_\theta(\textbf{\textit{x}})^\alpha 
d\textbf{\textit{x}}\Big].
\end{equation}
Suppose that $X_1,\dots,X_n$ is an i.i.d. sample drawn according to $p_\mu$ in
(\ref{1eqn:student_t_for_alpha_equal_to_2}). Then
\begin{align}
\label{1eqn:likelihood_function_of_student_t_for_alpha_2}
\lefteqn{L^{(2)}(\mu\mid X_1,\dots,X_n)}\nonumber\\ &= 2 \log \Big[\tfrac{1}{n} \sum\limits_{i=1}^n p_\mu (X_i) {\bf 1}(\mu -\sqrt{5}
\leq X_i \leq
\mu +\sqrt{5})\Big]\nonumber\\
&\hspace*{2.3cm} - \log \Big(\mathbb{E}_\mu \Big[N_{2} \Big\lbrace 1 -\frac{(X -\mu)^2}{5}\Big\rbrace\Big]\Big)\nonumber\\
&= 2 \log \Big[\tfrac{1}{n} \sum\limits_{i=1}^n p_\mu (X_i) {\bf 1}(X_i -\sqrt{5}\leq \mu \leq
X_i +\sqrt{5})\Big]\nonumber\\
&\hspace*{4.6cm} - \log \frac{4 N_2}{5},
\end{align}
where ${\bf 1}(\cdot)$ denotes the indicator function. 

The maximizer of $L^{(2)}(\mu)$ is same as the maximizer of
\begin{eqnarray}
\label{1eqn:alternative_likelihood_of_student_t_for_alpha_2}
\ell^{(2)} (\mu) := \sum\limits_{i=1}^n p_\mu (X_i) {\bf 1}(X_i -\sqrt{5}\leq \mu \leq
X_i +\sqrt{5}).
\end{eqnarray}
Without loss of generality, let us assume that $X_1< X_2<\cdots<X_n$.
It is clear from (\ref{1eqn:alternative_likelihood_of_student_t_for_alpha_2})
that one needs to have the knowledge of the entire sample to decide the maximizer of 
$\ell ^{(2)}(\mu)$.
\begin{equation}
\label{1eqn:best_case}
  \text{Let us first suppose that }(X_n - X_1) \leq 2\sqrt{5}.
\end{equation}
Then one can choose
a $\mu$ in $[ X_1 -\sqrt{5}, X_n +\sqrt{5}]$ such that $p_\mu(X_i)> 0$ for some 
$i\in \lbrace 1,\dots,n\rbrace$. Thus, in this case, we have
\begin{displaymath}
	\ell^{(2)}(\mu) = \left\{
	\begin{array}{ll}
	 { p_\mu(X_1),~ \text{for}~ \mu \in [X_1-\sqrt{5}, X_2-\sqrt{5}] }  
	\\
	  {\sum\limits_{i=1}^2 p_\mu(X_i),~\text{for}~ \mu\in [X_2-\sqrt{5}, X_3-\sqrt{5}]}  
	  \\
	  ~~\vdots\\
	 {\sum\limits_{i=1}^{n-1} p_\mu(X_i),~\text{for}~ \mu \in [X_{n-1}-\sqrt{5}, X_{n}-\sqrt{5}]}  
	  \\
	 {\sum\limits_{i=1}^{n} p_\mu(X_i),~\text{for}~ \mu\in [X_{n}-\sqrt{5}, X_{1}+\sqrt{5}]}\\
	 {\sum\limits_{i=2}^{n} p_\mu(X_i),~\text{for}~ \mu\in [X_{1}+\sqrt{5}, X_{2}+\sqrt{5}]}  
	   \\
	  ~~\vdots\\
	 {\sum\limits_{i=n-1}^{n} p_\mu(X_i),\text{for}~\mu\in [X_{n-2}+\sqrt{5}, X_{n-1}+\sqrt{5}]}\\
	 {p_\mu(X_n),~\text{for}~ \mu\in [X_{n-1}+\sqrt{5}, X_{n}+\sqrt{5}]}\\
	 0 \hspace*{1.2cm}\text{otherwise}
	 .
	\end{array}
	\right.
	\end{displaymath}
Let $\overline{X^{(k)}}:= \frac{1}{k}\sum
_{i=1}^k X_i$, and $\overline{X_{(k)}}:= \frac{1}{n-k}\sum
_{i=k+1}^n X_i$ for $k=1,\dots,n-1$.

The maximizer of $\ell^{(2)}(\mu)$ on $[X_k-\sqrt{5}, X_{k+1}-\sqrt{5}]$ is
\begin{align*}
\mu^{(k)} := \text{median} \lbrace X_k-\sqrt{5}, \overline{X^{(k)}}, X_{k+1}-\sqrt{5}\rbrace,
\end{align*}
that on $[X_n-\sqrt{5}, X_1+\sqrt{5}]$ is
\begin{align*}
\mu^{(n)} := \text{median} \lbrace  X_n-\sqrt{5}, \overline{X}, X_{1}+\sqrt{5}\rbrace, 
\end{align*}
and on $[X_k+\sqrt{5}, X_{k+1}+\sqrt{5}]$ is
\begin{align*}
\mu_{(k)} := \text{median} \lbrace X_k+\sqrt{5}, \overline{X_{(k)}}, X_{k+1}+\sqrt{5}\rbrace,
\end{align*}
for $k=1,\dots,n-1$. Let
\begin{equation*}
\mathcal{M} := \lbrace \mu^{(k)}, \mu^{(n)}, \mu_{(k)}: k=1,\dots,n-1\rbrace.
\end{equation*}
Then the estimator of $\mu$ is
\begin{equation*}
\widehat{\mu}:= \displaystyle\arg\max_{\mu\in \mathcal{M}} \ell^{(2)} (\mu).
\end{equation*}
Thus it is clear that $\widehat{\mu}$ is not necessarily $\overline{X}$. For illustration,
let us suppose that the observed sample is 4.6, 4.7, 6.0, 7.0, 8.2, 8.6, 8.7, 8.8, 8.9, and 9.0. Then
$X_n -X_1 = 9-4.6= 4.4 <2\sqrt{5}$, $\mu^{(10)}= 6.84$ and
\vspace*{-0.2cm}

\begin{table}[h!]
  \begin{center}
    \begin{tabular}{|l|l|l|l|l|l|l|l|l|}
    \hline
      $\mu^{(1)}$ & $\mu^{(2)}$ & $\mu^{(3)}$ & $\mu^{(4)}$ & $\mu^{(5)}$ &
      $\mu^{(6)}$ & $\mu^{(7)}$ & $\mu^{(8)}$ & $\mu^{(9)}$\\
      \hline
     2.46 & 3.76 & 4.76 & 5.57  & 6.1 &
      6.46 & 6.56 & 6.66 & 6.76 \\ 
      \hline
      \hline
      $\mu_{(1)}$ & $\mu_{(2)}$ & $\mu_{(3)}$ & $\mu_{(4)}$ & $\mu_{(5)}$ &
      $\mu_{(6)}$ & $\mu_{(7)}$ & $\mu_{(8)}$ & $\mu_{(9)}$\\
      \hline
     6.94 & 8.15  & 8.46  & 9.24  & 10.44 &
     10.84  & 10.94 & 11.04 &  11.14\\ 
      \hline
    \end{tabular}
  \end{center}
\end{table}
\vspace*{-0.4cm}
The respective maximum values of $\ell^{(2)} (\mu)$ are given by $3.7 N_2$ and
\vspace*{-0.3cm}

\begin{table}[h!]
  \begin{center}
    \begin{tabular}{|l|l|l|l|l|l|}
    \hline
     0.08$N_2$ & 1.68$N_2$ & 2.69$N_2$ & 3.21$N_2$ & 3.11$N_2$ &
     3.07$N_2$\\
     \hline  
     3.15$N_2$ & 3.3$N_2$ & 3.5$N_2$ &
     4.02$N_2$ & 6.37$N_2$ & 6.42$N_2$\\
      \hline
      5.57$N_2$  & 2.3$N_2$ &
     0.82$N_2$  & 0.5$N_2$ & 0.25$N_2$ & 0.84$N_2$ \\ 
      \hline
    \end{tabular}
  \end{center}
\end{table}
\vspace*{-0.4cm}
Thus, the maximum value of $\ell^{(2)} (\mu)$ is $6.42N_2$ and the maximizer is $\mu_{(3)} = 8.46$. Hence 
$\widehat{\mu} = 8.46$, which is not equal to $\overline{X}=7.45$.

Similarly, if (\ref{1eqn:best_case}) is true excluding one sample, say $X_1$, that is,
if $X_n - X_2 \leq 2 \sqrt{5}$,
but $X_n- X_1 > 2\sqrt{5}$, then we can
follow the same procedure with $X_2,\dots,X_n$ to find $\widehat{\mu}$. Thus, in general, if
there are $k$ samples such that (\ref{1eqn:best_case}) is true excluding these $k$ samples, then
we can proceed similarly with the rest $(n-k)$ samples to find $\widehat{\mu}$. Finally, when
all the samples are more than $2\sqrt{5}$ apart from each other, the intervals in
(\ref{1eqn:alternative_likelihood_of_student_t_for_alpha_2}) will be disjoint and
in this case any sample point can be taken to be the estimator.

\section{Summary}
\label{sec:summary}
In this paper we extended the already known projection theorem of the $\mathscr{I}_\alpha$-divergence
to the canonical $\mathbb{M}^{(\alpha)}$-family on the finite alphabet set of
\cite{KumarS15J2} to the more general multivariate,
continuous $\mathbb{M}^{(\alpha)}$-family and applied the result to find
estimators for the Student-t distributions. In the case when $\alpha <1$, we showed that the estimators are same as the maximum likelihood estimators of the Gaussian distribution, and can be obtained by solving the estimating equations (or projection equations). However, in the case when $\alpha >1$, the estimators cannot be obtained by solving the estimating equations and one needs to obtain it by maximizing the generalized likelihood function on a case-by-case basis. 

\section*{Acknowledgments}
Atin Gayen is supported by an INSPIRE fellowship, the Department of Science and Technology, Government of India.

 \end{document}